\documentclass{amsart}

\usepackage{graphicx}
\usepackage{amssymb}

\usepackage{amsthm}
\usepackage{mathrsfs}

\usepackage{amsmath}

\usepackage{comment}

\usepackage{lineno}
\usepackage{color}
\newtheorem{theorem}{Theorem}
\newtheorem{proof}{Proof}
\begin{document}

\title[Simple Second-Order Finite Differences for PDEs with Interfaces]{Simple Second-Order Finite Differences for Elliptic PDEs with Discontinuous Coefficients and Interfaces}

\author{Chung-Nan Tzou \and Samuel N. Stechmann}
\subjclass[2010]{65M06, 76T99, 35J05}
\keywords{sharp interface,
immersed boundary method, 
immersed interface method, 
ghost fluid method, 
jump conditions, 
phase changes}


\begin{abstract}
In multi-phase fluid flow, fluid-structure interaction, and other applications,
partial differential equations (PDEs) often arise with discontinuous coefficients
and singular sources (e.g., Dirac delta functions).
These complexities arise due to changes in material properties at 
an immersed interface or embedded boundary, which may have an irregular shape.
Consequently, the solution and its gradient 
can be discontinuous, and numerical methods can be difficult to design.
Here a new method is presented and analyzed, using 
a simple formulation of one-dimensional finite differences 
on a Cartesian grid, allowing for a relatively easy setup for 
one-, two-, or three-dimensional problems.
The derivation is relatively simple and mainly involves centered finite difference formulas, 
with less reliance on the Taylor series expansions of typical immersed interface method derivations.
The method preserves a sharp interface with discontinuous solutions,
obtained from a small number of iterations (approximately five) of solving 
a symmetric linear system with updates to the right-hand side.
Second-order accuracy is rigorously proven in one spatial dimension
and demonstrated through numerical examples in two and three spatial dimensions.
The method is tested here on the variable-coefficient Poisson equation,
and it could be extended for use on time-dependent problems
of heat transfer, fluid dynamics, or other applications.
\\

\end{abstract}

\maketitle

\section{Introduction}
\label{sec:intro}
In many applications,
partial differential equations (PDEs) arise with discontinuous coefficients
and singular sources (e.g., Dirac delta functions).
These complexities often arise due to changes in material properties at an interface
or immersed boundary, which may have an irregular shape;
see Fig.~\ref{Domain_Diag}.
For example, the immersed boundary may be a rigid or flexible structure,
such as a heart valve
\cite{glmp09},
or the immersed interface may separate two fluids as in gas bubbles or liquid droplets
\cite{sfso98}.
Our own interest was motivated by recently derived equations for atmospheric dynamics,
in the limit of rapid rotation and strong (moist) stratification, 
including phase changes of water and phase interfaces between cloudy and non-cloudy regions
\cite{ss17}.

\begin{figure}[!htbp]
    \centering
        \includegraphics[width=\textwidth]{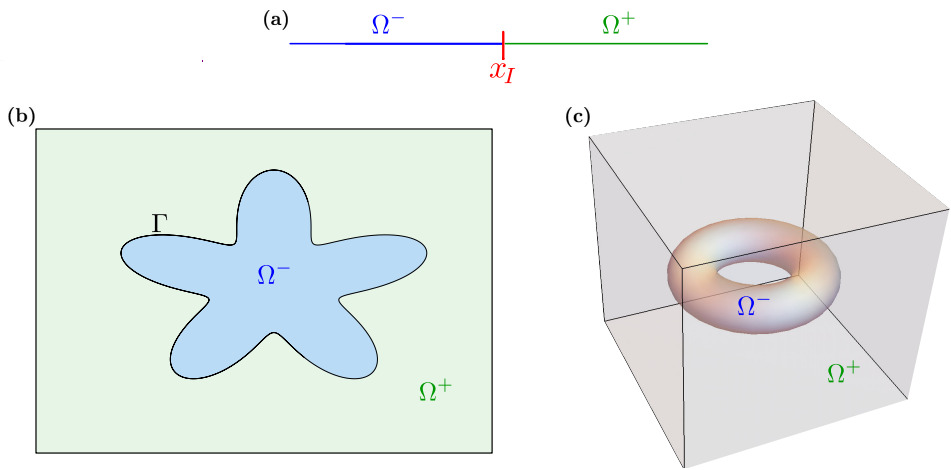}
        \caption{Examples of interfaces separating two regions $\Omega^-$ and $\Omega^+$
                 in (a) 1D, (b) 2D, and (c) 3D.}
         \label{Domain_Diag}
\end{figure}

For PDEs with such complexities, numerical methods can be challenging 
to design. Elliptic PDEs are a common test case,
and they often form an important component of time-dependent systems.
Many methods have been proposed using
finite element methods
\cite{betal10,hetal12},
finite volume methods
\cite{jly06,ccg11},
and finite difference methods.
Each of these approaches can be valuable in different situations,
depending on priorities of computational efficiency, ease of implementation, etc.
A primary goal of the present paper is simplicity,
and finite difference methods, with Cartesian grids, are perhaps
the simplest class of methods.
Therefore, for comparison, we next describe some finite difference methods in more detail.

The immersed boundary method (IBM)
was introduced in the pioneering work of 
Peskin
\cite{p72,p77jcp,p02}.
The IBM is simple and efficient and has been applied
to a variety of problems with three-dimensional
fluid flow
\cite{glmp09,kbgd16}.
In the IBM approach, the effect of the immersed boundary
is represented as a forcing function applied to the fluid.
Ideally, the forcing should be singular and the solution should
have discontinuities. However, the IBM uses
a smoothed version of a Dirac delta function,
which introduces some smearing near the boundary or interface
and causes the solution to be continuous.
The method was originally designed with first-order accuracy, and
it has been extended to be ``formally'' second-order accurate
\cite{lp00,gp05,mp08,fetal13i,fetal14ii},
although the ``formal'' second-order accuracy holds only in the case
that the forcing is sufficiently smooth, not in the case of a nearly
singular forcing.

The immersed interface method (IIM)
was developed to produce improvements such as
second-order accuracy and
a solution with a sharp discontinuity and no smearing at the interface
\cite{ll94,li96}.
The method is derived by allowing an extended stencil,
beyond the standard stencil for the Laplacian operator,
to be used at grid points near the interface;
for the extended stencil,
the finite-difference weights are then found by the
method of undetermined coefficients, with constraints on the coefficients
being chosen to achieve the desired local truncation error based on Taylor series.
The extended stencil of the IIM must be chosen with care
in order to avoid instability
\cite{fk01,li01,dil03},
since the IIM linear operator is not symmetric.
One approach is to carefully construct the IIM operator to satisfy
a discrete maximum principle by using constrained quadratic
optimization techniques
\cite{li01,dil03}.

While the IIM has been implemented in multi-dimensional fluid flow problems,
the formulation is complicated
by the need for derivations of many spatial and temporal jump conditions,
and also derivatives of jump conditions
\cite{ll01,ll03,xw06siam,xw06jcp}.
Many other versions of the IIM with different derivations have been developed 
\cite{wiegmann2000explicit,b04,rw03,lai2008simple}, and some are discussed further 
in~\S\ref{sec:comparisons} below.  In the present paper, one distinguishing feature 
is that the present derivation involves the relatively simple use of centered finite difference formulas, 
without the need for derivatives of jump conditions, and with less reliance on the 
Taylor series expansions of typical IIM derivations. Such simplifications to the derivations 
should contribute to enhanced ease of use on three-dimensional problems.

The ghost fluid method (GFM) is another method that produces 
a solution with a sharp discontinuity and no smearing at the interface
\cite{lfk00}.
While it is only first-order accurate, the GFM is
simple to formulate and implement, and it is efficient for problems with
three-dimensional multi-phase fluid flow
\cite{kfl00,sshoz07}.
Another advantageous property is that the GFM finite difference operator
is symmetric, which allows the use of conjugate gradient algorithms
and guarantees robustness of the method.

In the present paper,
the goal is to design a method with the advantageous properties of
the GFM -- sharp interface, easy to formulate and implement, 
efficient for use on three-dimensional problems,
and utilization of a symmetric matrix -- while also achieving the possibility of
second-order accuracy.
The simple formulation here (\S\ref{sec:methods}) 
uses elementary finite differences along
one-dimensional coordinates, and the resulting linear system can be written 
with the same symmetric matrix as the GFM but with corrections to the
right-hand side that yield second-order accuracy.
The right-hand-side corrections are determined iteratively,
which is the main new computational expense beyond the GFM.
Note that, while this interesting algorithmic connection exists with the GFM,
the derivations of the GFM and the present method are quite different;
the present method is derived using finite differences (with explicit estimates
of local truncation error from finite difference formulas), 
whereas the GFM and its error and convergence are 
based on a weak formulation of the problem
\cite{ls03mc}.
Example solutions with the present method are shown for one-dimensional (1D),
two-dimensional (2D), and three-dimensional (3D) problems
(\S\ref{sec:examples}).
A small, fixed number of iterations ($\approx 5$) is shown to be
sufficient for achieving a second-order accurate solution
(\S\ref{sec:stopping}),
which suggests the present methods may be efficient enough for use on
complex three-dimensional fluid flow.
Conclusions and further comparisons with the formulations of other methods 
\cite{wiegmann2000explicit,b04,rw03,lai2008simple,mnr11,mnr17}
are discussed in~\S\ref{sec:comparisons} and \S\ref{sec:conclusions}.

Given that many previous methods have been proposed for this problem over 
many years, it is worthwhile to emphasize one of the main distinguishing 
features of the present method: a simple derivation and setup. The derivation here 
is mainly achieved using centered finite difference formulas, so it is
relatively easy to formulate and set up the method, even in 3D.  
At the same time, the method does utilize a small number of iterations,
so it may have a greater computational expense than some other methods
(unless one could propose a more sophisticated and faster iterative procedure,
a direction which we have not yet pursued exhaustively).
In summary, in terms of practical use, 
the simple derivation and formulation should be useful
for applications where one is less concerned with achieving the 
least possible expense of the computation itself and more concerned with minimizing
the time and effort needed to initially design and code the method.

\section{Numerical methods}
\label{sec:methods}

In this section, the numerical methods are derived
for 1D, 2D, and 3D
equations in sections
\ref{sec:method-1d},
\ref{sec:method-2d}, and
\ref{sec:method-3d}, respectively.
A rigorous proof of second-order convergence is presented in section
\ref{sec:proof}
for the 1D case.

\subsection{One dimension}
\label{sec:method-1d}
Consider a one (spatial) dimensional domain $\Omega$ divided into subdomains $\Omega^+$ 
and $\Omega^-$ by an interface $\Gamma$. The variable coefficient Poisson 
equation on each subdomain reads
\begin{equation}\label{1D_eqn}
\left(\beta u_x\right)_x=f(x), \qquad\mbox{ for }x\in\Omega\setminus\Gamma,
\end{equation}
where $\beta=\beta(x)$ and $f(x)$ can be discontinuous across interface points $x_I\in\Gamma$.
The jump conditions across the interface are given as
\begin{equation}\label{jumpcond}
\begin{array}{llllll}
[u]=u^+-u^-=a(x), &\mbox{ for }x\in \Gamma, \\
{[\beta u_x]}=\beta^+ u_x^+-\beta^- u_x^-=b(x), &\mbox{ for }x\in \Gamma.
\end{array}
\end{equation}
We focus here on the case of two subdomains and one interface point,
as it is straightforward to extend the methods for cases with
more subdomains and interface points.

As an alternative formulation of the problem, one could incorporate the jump conditions
(\ref{jumpcond})
into the differential itself by adding singular sources to the right-hand side of the equation.
In such a formulation, the differential equation would take the form
$(\beta u_x)_x = f(x)+b_I\delta(x-x_I)+a_I\bar{\beta}\delta'(x-x_I)$,
where $\bar{\beta}=(\beta^+ + \beta^-)/2$ and $a_I=a(x_I)$ and $b_I=b(x_I)$,
and where this differential equation is valid over the entire domain $\Omega$.
On the other hand, the differential equation in (\ref{1D_eqn}) is valid only within each of the separate regions
$\Omega^+$ and $\Omega^-$, and the jump conditions in (\ref{jumpcond}) are needed to
connect the solutions in $\Omega^+$ and $\Omega^-$ and complete the problem specification.
It will be convenient here to use the separate formulation in
(\ref{1D_eqn})--(\ref{jumpcond})
throughout the paper.

\subsubsection{Finite differences}\label{1DFiniteDiff}
A second-order finite-difference method can be 
derived on a Cartesian grid, with a symmetric operator, in the following way. 

First, if the interface $\Gamma=\{x_I\}$ does not intersect with the grid edges connecting the three points $x_{i-1}$, $x_i$, and $x_{i+1}$, then we call $x_i$ a standard Cartesian point. For all the standard Cartesian points we follow the standard second-order discretization for (\ref{1D_eqn}):
\begin{equation}\label{standarddiscretization}
\frac{\beta_{i+\frac{1}{2}} \left(\frac{u_{i+1}-u_i}{\Delta x}\right)-\beta_{i-\frac{1}{2}} \left(\frac{u_i-u_{i-1}}{\Delta x}\right)}{\Delta x}=f_i+O(\Delta x^2).
\end{equation}

Next, consider nonstandard Cartesian points, such as $x_i$ and $x_{i+1}$ with an interfacial 
point $x_I\in \Gamma$ in between and with $x_i\in \Omega^-$ and $x_{i+1}\in \Omega^+$,
as shown in Fig.~\ref{1-D_FD}.
Since the number of nonstandard points is assumed to be small,
it should be possible to have an overall second-order-accurate method
that locally uses a first-order discretization at nonstandard points.
Therefore, we use a first-order discretization of $(\beta u_x)_x$,
\begin{equation}
(\beta u_x)_x(x_i)
=
\frac{\beta(x_{m-}) u_x(x_{m-})-\beta(x_{i-\frac{1}{2}}) u_x(x_{i-\frac{1}{2}})}{x_{m-}-x_{i-\frac{1}{2}}}
+O(\Delta x),
\end{equation}
followed by second-order discretizations of the $u_x$ terms, which lead to
\begin{equation}\label{primitive_uxx}
\frac{\beta_{m-}\frac{u_{I-}-u_i}{(1-\theta)\Delta x}-\beta_{i-\frac{1}{2}}\frac{u_i-u_{i-1}}{\Delta x}}{\frac{2-\theta}{2}\Delta x}
=f_i+O(\Delta x),
\end{equation}
where $\theta=(x_{i+1}-x_I)/\Delta x$.
Note that the midpoints $x_{m-}=(x_i+x_I)/2$ and $x_{m+}=(x_I+x_{i+1})/2$, 
illustrated in
Fig.~\ref{1-D_FD}, 
are useful here to allow second-order discretizations of $u_x$.

\begin{figure}[!htbp]
    \centering
        \includegraphics[width=0.8\textwidth]{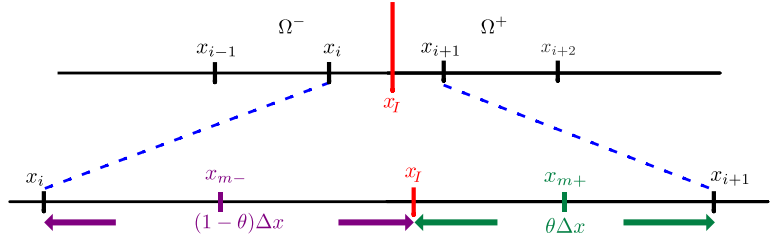}
        \caption{Cartesian grid points and an interfacial point in between.}\label{1-D_FD}
\end{figure}

The final step is to replace in (\ref{primitive_uxx}) the appearance of the interface value $u_{I-}$ with Cartesian values and adjustments consisting of known quantities. To do this, we obtain additional equations by discretizing (\ref{1D_eqn}) at 
$x_{I-}$ and $x_{I+}$, the left and right limit of $x_I$,
using a method similar to the one above:
\begin{equation}
\mbox{At } x_{I+}: \quad \frac{\beta_{m+}\frac{u_{i+1}-u_{I+}}{\theta \Delta x}-\beta_{I+}u_x(x_{I+})}{\theta \Delta x/2}=f_{I+}+O(\Delta x)\label{zpeqn}
\end{equation}
and 
\begin{equation}
\mbox{At } x_{I-}: \quad \frac{\beta_{I-}u_x(x_{I-})-\beta_{m-}\frac{u_{I-}-u_i}{(1-\theta) \Delta x}}{(1-\theta) \Delta x/2}=f_{I-}+O(\Delta x).\label{zmeqn}
\end{equation}
The non-Cartesian unknowns $u_x(x_{I\pm})$ and $u_{I\pm}$ above can now be replaced by Cartesian unknowns by the following two steps.
First, the weighted sum $(\theta\Delta x/2)\cdot(\ref{zpeqn})+((1-\theta)\Delta x/2)\cdot(\ref{zmeqn})$ is a combination that
produces the jump $[\beta u_x]$:
\begin{equation}\label{int_cond}
  	\begin{gathered}
\beta_{m+}\left(\frac{u_{i+1}-u_{I+}}{\theta\Delta x}\right)-\beta_{m-}\left(\frac{u_{I-}-u_i}{(1-\theta)\Delta x}\right)-[\beta u_x]\hfill \\
=\Big(\theta\cdot f_{I+}+(1-\theta)\cdot f_{I-}\Big)\frac{\Delta x}{2}+O(\Delta x^2).
	\end{gathered}
\end{equation}

Second, by using the jump conditions (\ref{jumpcond}),
we see that (\ref{int_cond}) can be rewritten as our desired formula for
replacing $u_{I-}$ by Cartesian $u$ values:
\begin{equation}\label{uintformula}
\begin{aligned}
u_{I-}
&=\frac{\hat\beta(1-\theta)}{\beta_{m-}}u_{i+1}
+\frac{\hat\beta\theta}{\beta_{m+}}u_{i}\hfill\\
&-\frac{\hat\beta \theta(1-\theta)\Delta x^2}{\beta_{m+}\beta_{m-}}\left(\frac{\beta_{m+}a_I}{\theta\Delta x^2}+\frac{b_I}{\Delta x}+\frac{1}{2}\Big(\theta\cdot f_{I+}+(1-\theta)\cdot f_{I-}\Big)\right),
\end{aligned}
\end{equation}
where 
\begin{equation}
\hat\beta=\frac{\beta_{m+}\beta_{m-}}{(1-\theta)\cdot\beta_{m+}+\theta\cdot\beta_{m-}}.
\end{equation}
Lastly, substituting (\ref{uintformula}) into (\ref{primitive_uxx})
yields a first-order discretization of the differential equation at $x_i$, 
in terms of only Cartesian values of $u$:
\begin{eqnarray}\label{symmetric_uxx}\notag
&&\frac{1}{\Delta x^2}\left(\beta_{i-\frac{1}{2}}\cdot u_{i-1}-\left(\beta_{i-\frac{1}{2}}+\hat\beta\right)u_i+\hat\beta\cdot u_{i+1}\right)\\
&&\qquad 
= f_i\cdot\left(\frac{2-\theta}{2}\right)
+\frac{\hat\beta \theta}{\beta_{m+}}\left(\frac{\beta_{m+}}{\theta}\frac{a_I}{\Delta x^2}+\frac{b_I}{\Delta x}+\frac{1}{2}\Big(\theta\cdot f_{I+}+(1-\theta)\cdot f_{I-}\Big)\right).
\end{eqnarray}
For the neighboring nonstandard point at $x_{i+1}$, one can derive
a similar finite difference formula:
\begin{eqnarray}
&& \frac{1}{\Delta x^2}\left(\hat\beta\cdot u_{i}-\left(\hat\beta+\beta_{i+\frac{3}{2}}\right)u_{i+1}+\beta_{i+\frac{3}{2}}\cdot u_{i+2}\right)= f_{i+1}\cdot\left(\frac{1+\theta}{2}\right)
\nonumber \\
&&\qquad 
+\frac{\hat\beta (1-\theta)}{\beta_{m-}}\left(-\frac{\beta_{m-}}{(1-\theta)}\frac{a_I}{\Delta x^2}+\frac{b_I}{\Delta x}+\frac{1}{2}\Big(\theta\cdot f_{I+}+(1-\theta)\cdot f_{I-}\Big)\right).
\label{symmetric_uxx2}
\end{eqnarray}
Comparing (\ref{symmetric_uxx}) and (\ref{symmetric_uxx2}), it is clear that the difference operator acting on $u$ is symmetric.
The linear system can be solved using many standard efficient methods.

Note that this method in (\ref{symmetric_uxx})--(\ref{symmetric_uxx2})
looks similar to the GFM, which is first-order accurate
\cite{lfk00,ls03mc},
but (\ref{symmetric_uxx})--(\ref{symmetric_uxx2})
include important differences that render this method
second-order accurate.
For instance, the right-hand-side terms in
(\ref{symmetric_uxx})--(\ref{symmetric_uxx2})
have coefficients that are different from the GFM
and that arise here as part of a systematic finite-differences
derivation. Also, the values of $\beta$ at the midpoints
$x_{m-}$ and $x_{m+}$ 
were needed for the present method,
whereas $\beta$ values at the interface and Cartesian grid points and Cartesian midpoints
are utilized in the GFM
\cite{lfk00,ls03mc}.

In comparison to the IIM
\cite{ll94},
notice that the present method has a symmetric operator,
whereas the IIM operator is non-symmetric.
Also, the derivation of the IIM requires taking derivatives
of jump conditions, whereas the present method is derived
by simply applying finite difference formulas to the differential equation.

To summarize, the basic idea in deriving 
(\ref{symmetric_uxx})--(\ref{symmetric_uxx2})
was to (i) start with midpoint-based finite differences
using both Cartesian points and interface points,
and then (ii) use the jump conditions to eliminate the
interface values $u_{I\pm}$ from the system.

\subsubsection{Proof of second-order convergence}
\label{sec:proof}

\vspace{12pt}
\begin{theorem}
The numerical solution in~\S\ref{1DFiniteDiff} converges to the exact 
solution in the $L^2$ norm with second-order accuracy:
$||\mathbf{U}-\mathbf{U}_{ex}||_2=O(\Delta x^2)$.
\end{theorem}
\begin{proof}
The setup of the proof is as follows.
The numerical method in (\ref{standarddiscretization}), (\ref{symmetric_uxx}) and (\ref{symmetric_uxx2}) can be written in matrix-vector form as
$A\mathbf{U}=\mathbf{F}$,
and the exact solution satisfies
$A\mathbf{U}_{ex}=\mathbf{F}+\boldsymbol{\tau}$,
where $\boldsymbol{\tau}$ is the local truncation error.
The error $\mathbf{e}=\mathbf{U}-\mathbf{U}_{ex}$ then satisfies
$A\mathbf{e}=-\boldsymbol{\tau}$,
and solving for $\mathbf{e}$ gives
$\mathbf{e}=-A^{-1}\boldsymbol{\tau}$.
The $L^2$ norm of the error then satisfies
\begin{equation}
\|\mathbf{e}\|_2=\|A^{-1}\boldsymbol{\tau}\|_2 \le \|A^{-1}\|_2 \|\boldsymbol{\tau}\|_2,
\label{eqn:convergence-proof}
\end{equation}
where the remaining task is to analyze $\|A^{-1}\|_2$ and $\|\boldsymbol{\tau}\|_2$ for small $\Delta x$.

Consistency was established in~\S\ref{1DFiniteDiff}.
Specifically, the local truncation error can be written as
\begin{equation}
\boldsymbol{\tau}=\boldsymbol{\tau}_s+\boldsymbol{\tau}_{ns},
\quad\mbox{with}\quad
\|\boldsymbol{\tau}_s\|_2=O(\Delta x^2),
\quad
\|\boldsymbol{\tau}_{ns}\|_2=O(\Delta x^2),
\label{eqn:consistency-proof}
\end{equation}
where we have split $\boldsymbol{\tau}$ so that
the elements of $\boldsymbol{\tau}_s$ are nonzero only at
standard points and the elements of $\boldsymbol{\tau}_{ns}$
are nonzero only at non-standard points.
The $O(\Delta x^2)$ scaling in (\ref{eqn:consistency-proof}) 
is then true because
each element of $\boldsymbol{\tau}_s$ is $O(\Delta x^2)$, based on the finite difference formulas
at the standard points;
and each element of $\boldsymbol{\tau}_{ns}$ is $O(\Delta x)$, 
but the fraction of
non-standard points is $O(\Delta x)$, so $\|\boldsymbol{\tau}_{ns}\|_2=O(\Delta x^2)$.

Stability is established by the bound
\begin{equation}
\|A^{-1}\|_2\le \frac{|\Omega|^2}{\beta_m},
\label{eqn:stability-proof}
\end{equation}
where 
$|\Omega|$ is the total length of the domain and $\beta_m=\min_{x\in\Omega} \beta(x)$ is a constant that is independent of $\Delta x$,
and it is assumed that $\beta(x)>0$ for all $x$.
The proof of this bound is well-known
\cite{shlui12}
and is based on summation by parts and discrete Poincar\'e-Friedrichs inequality.

The proof of the theorem
is completed by combining the consistency and stability results in
(\ref{eqn:consistency-proof}) and (\ref{eqn:stability-proof})
to show that (\ref{eqn:convergence-proof}) is $O(\Delta x^2)$.
\end{proof}

Note that we have no such proof in 
two- or three-dimensional space, although proofs for 2D and 3D have been presented for
similar methods
\cite{bl06},
and numerical examples below demonstrate second-order convergence.

\subsection{Two dimensions}
\label{sec:method-2d}

Now consider the two-dimensional Poisson equation
\begin{equation}\label{2deqn}
(\beta u_x)_x+(\beta u_y)_y=f(x,y) \qquad \mbox{for} \quad \Omega\setminus\Gamma,
\end{equation}
where $\Omega=\Omega^+\cup\Omega^-\cup\Gamma$ and $\Gamma$ is the interface between the
sets $\Omega^+$ and $\Omega^-$. 
With ${\bf n}=(n^1(x,y),n^2(x,y))$ as the unit normal along $\Gamma$, 
the interface jump conditions are given as
\begin{equation}\label{2djumpcond}
\begin{array}{llllll}
[u]=u^+-u^-=a({\bf x}), &\mbox{ for }{\bf x}\in \Gamma, \\
{[\beta u_n]}=\beta^+ u_n^+-\beta^- u_n^-=b({\bf x}), &\mbox{ for }{\bf x}\in \Gamma,
\end{array}
\end{equation}
where $u_n=\mathbf{n}\cdot\nabla u$
is the derivative of $u$ in the direction of the normal vector. 

\subsubsection{Finite differences}
\label{sec:fin-diff-2d}

The goal of this section is to extend the ideas of the 1D case of \S\ref{sec:method-1d}
to the 2D case of
(\ref{2deqn})--(\ref{2djumpcond})
and arrive at a second-order finite-difference method.
Similar to the 1D case, we call a Cartesian point $(x_i,y_j)$
a standard point if this point and its nearest neighbors all lie within
$\Omega^+$ or all lie within $\Omega^-$.
For standard points, (\ref{2deqn}) is discretized with the standard, second-order, 
5-point finite-difference formula. 
For non-standard points, on the other hand, 
the interface must be taken into account.

\begin{figure}[!htbp]
    \centering
        \includegraphics[width=0.5\textwidth]{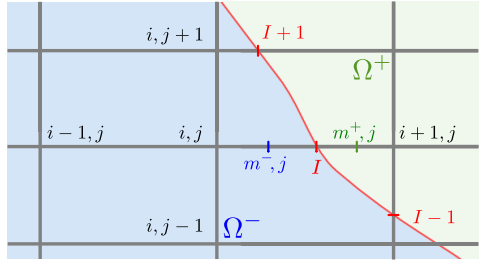}
        \caption{Non-standard grid point at $(x_i,y_j)$.}
         \label{fig:2dGrid}
\end{figure}

For nonstandard points, 
such as point $(x_i,y_j)$ illustrated in Fig.~\ref{fig:2dGrid},
we obtain a first-order discretization
by using similar ideas as in the 1D case.
Following a derivation similar to (\ref{primitive_uxx})--(\ref{symmetric_uxx}),
by essentially just replacing $f$ by $f-(\beta u_y)_y$,
we arrive at 
\begin{equation}
\begin{aligned}
 &\frac{1}{\Delta x^2}\left(\beta_{i-\frac{1}{2},j}\cdot u_{i-1,j}-\left(\beta_{i-\frac{1}{2},j}+\hat\beta\right)u_{i,j}+\hat\beta\cdot u_{i+1,j}\right)
\\
+ &\frac{1}{\Delta y^2}\left(\beta_{i,j-\frac{1}{2}}\cdot u_{i,j-1}-\left(\beta_{i,j-\frac{1}{2}}+\beta_{i,j+\frac{1}{2}}\right)u_{i,j}+\beta_{i,j+\frac{1}{2}}\cdot u_{i,j+1}\right)
\\
=&
f_{i,j}\cdot\left(\frac{2-\theta}{2}\right)
+(\beta u_y)_y(x_i,y_j)\cdot\frac{\theta}{2}
+F^x_{cor}+O(\Delta x),
\label{symmetric_2d_uxx}
\end{aligned}
\end{equation}
where 
\begin{equation}
\hat\beta=\frac{\beta(x_{m+},y_j)\cdot\beta(x_{m-},y_j)}{(1-\theta)\cdot\beta(x_{m+},y_j)+\theta\cdot\beta(x_{m-},y_j)},
\label{eqn:betahat}
\end{equation}
and
\begin{equation}
\begin{aligned}
F^x_{cor}&=
\frac{\hat\beta \theta}{\beta(x_{m+},y_j)}
\Bigg\{
\frac{\beta(x_{m+},y_j)a(x_I,y_j)}{\theta \Delta x^2}+\frac{[\beta u_x]}{\Delta x}\hfill \\
&+\frac{1}{2}\Big(\theta\cdot\left( f-(\beta u_y)_y\right)(x_{I+},y_j)+(1-\theta)\cdot\left( f-(\beta u_y)_y\right)(x_{I-},y_j)\Big)\Bigg\}.
\label{F_{cor}rection}
\end{aligned}
\end{equation}
This finite-difference formula has a left-hand side with the desirable property of a
symmetric operator,
as in the 1D case. However, the right-hand side of (\ref{symmetric_2d_uxx})
now depends on the solution $u$ itself, so an iterative method
will be described below for finding a solution.

Also, a more general case would allow for other interface crossings, such as 
a crossing at point
$(x_i,y_{J})$, with $y_j<y_J<y_{j+1}$,
which would generate some slight modifications to the derivation and finite-difference
formula. Since the more general case is only slightly different from
(\ref{symmetric_2d_uxx}),
it is relegated to 
appendix \ref{app:2D-two-crossings}.

To estimate the derivatives on the right-hand side of (\ref{symmetric_2d_uxx}),
simple finite differences are used.
For the term $(\beta u_y)_y(x_i,y_j)$,
standard centered differences can be used with the points
$(x_i,y_{j-1}), (x_i,y_j)$, and $(x_i,y_{j+1})$.
For the term $(\beta u_y)_y(x_{I-},y_j)$ at the interface, from (\ref{F_{cor}rection}),
one can approximate it with the nearby Cartesian value $(\beta u_y)_y(x_i,y_j)$
with an acceptable error of $O(\Delta x)$, and then one can use a standard centered discretization
with the points $(x_i,y_{j-1}), (x_i,y_j)$, and $(x_i,y_{j+1})$.
The term $(\beta u_y)_y(x_{I+},y_j)$ can be handled similarly by using $(\beta u_y)_y$ at
the nearby Cartesian point $(x_{i+1},y_j)$.
Lastly, the jump $[\beta u_x]$ from (\ref{F_{cor}rection}) can be written
in terms of normal and tangential jumps as
\begin{eqnarray}
[\beta u_x]
&=& [\beta u_n]n^1-[\beta u_\tau]n^2 
\nonumber \\
&=& b_I n^1-[\beta u_\tau]n^2.
\label{eqn:beta-ux-jump-2d}
\end{eqnarray}
The term $[\beta u_\tau]$ can then be estimated using finite differences
with $u$ values from the interface points labeled $I-1, I$, and $I+1$ in
Fig.~\ref{fig:2dGrid}
(or possibly using another triplet, say $I-2, I$, and $I+1$, if the two interface points 
$I-1$ and $I$ are located too close together, such as within $O(h^2)$ distance). 
Note that a second-order finite-difference formula is needed for $[\beta u_\tau]$
in order for the term $[\beta u_x]/\Delta x$ to have an error of $O(\Delta x)$.
To determine the $u$ values at the interface points, 
one can use the formula
\begin{equation}
\begin{aligned}
u(x_{I-},y_{j}) 
= 
\dfrac{(1-\theta)\hat\beta}{\beta(x_{m-},y_j)}u_{i+1,j}+\dfrac{\theta\hat\beta}{\beta(x_{m+},y_j)}u_{i,j}
-\dfrac{\hat\beta(1-\theta)\theta \Delta x^2}{\beta(x_{m+},y_j)\beta(x_{m-},y_j)}
\\
\cdot\Bigg(\dfrac{\beta(x_{m+},y_j)a_I}{\theta  \Delta x^2}+\frac{[\beta u_x]}{\Delta x}
+\frac{\theta}{2}\cdot((\beta u_{x})_x)_{i+1,j}+\frac{(1-\theta)}{2}\cdot((\beta  u_{x})_x)_{i,j}
\Bigg),
\label{eqn:uI-2d}
\end{aligned}
\end{equation}
and $u(x_{I_+},y_j)=u(x_{I_-},y_j)+a(x_I,y_j)$ by the jump condition (\ref{2djumpcond}).
This formula arises as part of the derivation of (\ref{symmetric_2d_uxx})
and is similar to the 1D case, and formulas for $u(x_i,y_{J_\pm})$ can be obtained similarly 
if the crossing is in the $y$-direction.
Note that this formula in 2D does not actually provide the desired result of the interface 
$u$ value in terms of the Cartesian $u$ values, since the right-hand side depends on interface $u$ values
via the $[\beta u_x]$ term. Nevertheless, this formula can be used as part of an iterative
procedure to complete the specification of the numerical methods.

\subsubsection{Iterative methods}
\label{sec:2d-iterative}

In this section, a simple iterative method is proposed here for solving the linear system from
\S\ref{sec:fin-diff-2d}. 

Before describing the standard iterative method of the present paper,
consider first a type of Picard iteration:
\begin{equation}
A\mathbf{u}^{[k+1]}=\mathbf{F}^{[k]}.
\label{eqn:picard-u}
\end{equation}
This is an iterative version of the matrix-vector form of the finite difference method,
one row of which is described in (\ref{symmetric_2d_uxx}):
$A$ is the symmetric matrix from the left-hand side,
$\mathbf{u}^{[k+1]}$ is the vector of all Cartesian $u$ values (from iteration $k+1$),
and $\mathbf{F}^{[k]}$ is the vector from the right-hand-side terms.
The basic idea is to iteratively update $\mathbf{F}^{[k]}$ on the right-hand side
as new, more accurate information about $\mathbf{u}^{[k]}$ is obtained.
As an initial condition, $\mathbf{F}^{[0]}$ is defined as the right-hand side of
(\ref{symmetric_2d_uxx}) with all instances of $u$ ignored,
and the first solution $\mathbf{u}^{[1]}$ is found by solving
$A\mathbf{u}^{[1]}=\mathbf{F}^{[0]}$.
As a result, \textit{the solution $\mathbf{u}^{[1]}$ at the first iteration is 
essentially the same as the GFM solution
\cite{lfk00,ls03mc}
and is therefore a first-order accurate solution.}
It can be used to estimate the interface $u$ values,
which we assemble abstractly into a vector $\mathbf{u}_I^{[k]}$
and update iteratively as
$\mathbf{u}_I^{[k+1]}=B\mathbf{u}_I^{[k]}+C\mathbf{u}^{[k+1]}+\mathbf{G}$,
one row of which is described by (\ref{eqn:uI-2d}):
the $B\mathbf{u}_I^{[k]}$ corresponds to the $[\beta u_\tau]$ term,
the $C\mathbf{u}^{[k+1]}$ corresponds to all terms with Cartesian $u$ values,
and the $\mathbf{G}$ corresponds to the jump terms involving $a_I$ and $b_I$.
An initial interface value of $\mathbf{u}_I^{[0]}=\mathbf{0}$
is used, consistent with the idea of ignoring all instances of $u$
in the initial condition $\mathbf{F}^{[0]}$. 
The second iteration then proceeds by defining $\mathbf{F}^{[1]}$
based on the right-hand side of (\ref{symmetric_2d_uxx})
and now using $\mathbf{u}^{[1]}$ and $\mathbf{u}_I^{[1]}$
to provide a more accurate estimate of the true $\mathbf{F}$ value.
The solution $\mathbf{u}^{[2]}$ at the second iteration is then found from solving
the symmetric system $A\mathbf{u}^{[2]}=\mathbf{F}^{[1]}$.
This procedure can be repeated to iteratively estimate the solution
of the finite-difference method.

For the stopping criterion for the iterative procedure, the differences
$u^{[k]}_{d}=\|\mathbf{u}^{[k+1]}-\mathbf{u}^{[k]}\|_\infty$
and
$F^{[k]}_{d}=\|\mathbf{F}^{[k+1]}-\mathbf{F}^{[k]}\|_\infty$
are monitored. 
When $k$ is large enough so that $u^{[k]}_{d}<h^2$, where $h=\Delta x=\Delta y$,
one can presumably stop iterating since the iterations are producing only small corrections
that are within the desired $O(h^2)$ accuracy of the numerical solution.
As our standard stopping criterion, in addition to $u^{[k]}_{d}<h^2$
we also require $F^{[k]}_{d}<h$ in order to ensure that the estimated
right-hand-side terms are not significantly changing at any location.
Note that, while this standard stopping criterion was chosen with solution accuracy 
as the main consideration, one could also imagine other stopping criteria that
consider computational efficiency or other factors; some other stopping criteria
are explored in \S\ref{sec:stopping}.

As the standard iterative method used here, a modification of Picard iteration is actually used.
While Picard iteration does work well in many cases, we found that it diverges in some cases.
Nevertheless, by making some slight modifications, a robust method can be designed.
Our standard iterative method here uses a simple relaxation procedure to extend Picard iteration;
it is described in \ref{relaxation}, and it is shown below to provide robust results.

\subsection{Three dimensions}
\label{sec:method-3d}

The three-dimensional Poisson equation is
\begin{equation}\label{3deqn}
(\beta u_x)_x+(\beta u_y)_y+(\beta u_z)_z=f(x,y,z), \qquad \mbox{for} \quad \Omega\setminus\Gamma,
\end{equation}
where $\Omega=\Omega^+\cup\Omega^-\cup\Gamma$ and $\Gamma$ is a surface that marks the interface between the
sets $\Omega^+$ and $\Omega^-$. 
The interface jump conditions are given as in the 2D case in (\ref{2djumpcond}).

The 3D discretization is essentially the same as in the 2D case in \S\ref{sec:method-2d}. 
We note one difference that arises: in 3D, the jump $[\beta u_x]$ from (\ref{eqn:beta-ux-jump-2d})
takes the form
\begin{eqnarray}
[\beta u_x] 
&=& [\beta u_n]c^0+[\beta u_{\tau_1}]c^1+[\beta u_{\tau_2}]c^2
\nonumber \\
&=& b_I c^0+[\beta u_{\tau_1}]c^1+[\beta u_{\tau_2}]c^2,
\label{eqn:beta-ux-jump-3d}
\end{eqnarray}
where $\hat{\mathbf{x}}=c^0\hat{\mathbf{n}}+c^1\hat{\boldsymbol{\tau}}_1+c^2\hat{\boldsymbol{\tau}}_2$
was used to write the unit coordinate vector $\hat{\mathbf{x}}$ in terms of the
interface normal vector $\hat{\mathbf{n}}$ 
and two unit vectors $\hat{\boldsymbol{\tau}}_1$ and $\hat{\boldsymbol{\tau}}_2$
from the 2D tangent plane of the interface.
Here, in 3D, note that tangential derivatives are needed in two independent directions in the 2D tangent plane.
The two directions can be conveniently chosen by using the Cartesian coordinate planes.
For example, if  $(x_I, y_j, z_k)\in \Gamma$, 
where $x_I$ is not a Cartesian grid point, then the intersection of surface $\Gamma$ 
and the plane $z=z_k$ can be used to define one direction in the 2D tangent plane,
and the intersection of surface $\Gamma$ and the plane $y=y_j$ can be used to define the other direction.
In this way, computation of the tangential derivatives in 3D can be reduced to essentially the same form
as in 2D.

\section{Examples}
\label{sec:examples}

In this section, second-order convergence is demonstrated through
numerical examples.
In all examples, the same grid spacing is used in each coordinate direction
($\Delta x=\Delta y=\Delta z$),
and the number of grid points in each coordinate direction is $N$,
so the total number of grid points is $N$, $N^2$, or $N^3$
for the 1D, 2D, or 3D cases, respectively.

\subsection{One dimension}

\subsubsection{Example 1D-1}
Consider a domain $\Omega=[0,1]$ separated into two sub-domains $\Omega-=[0,x_I)$ and $\Omega^+=(x_I,1]$, where $x_I=2-\sqrt{2}$. The solution to the one dimensional equation $\beta u_{xx}=f$ is $u^-=\exp(-x)-0.3646x+0.4$ and $u^+=exp(-x)/2+x^2/2+0.5005x$ where $\beta=100$ in $\Omega^-$ and $\beta=200$ in $\Omega^+$, with $f=100\exp(-x)$ in $\Omega^-$ and $f=100\exp(-x)+200$ in $\Omega^+$. The jump conditions connecting the two equations at $x_I$ are $a(x_I)=u^+-u^-=0$ and $b(x_I)=100(2u_x^+-u_x^-)=253.72$.

\begin{figure}[!htbp]
    \centering
        \includegraphics[width=\textwidth]{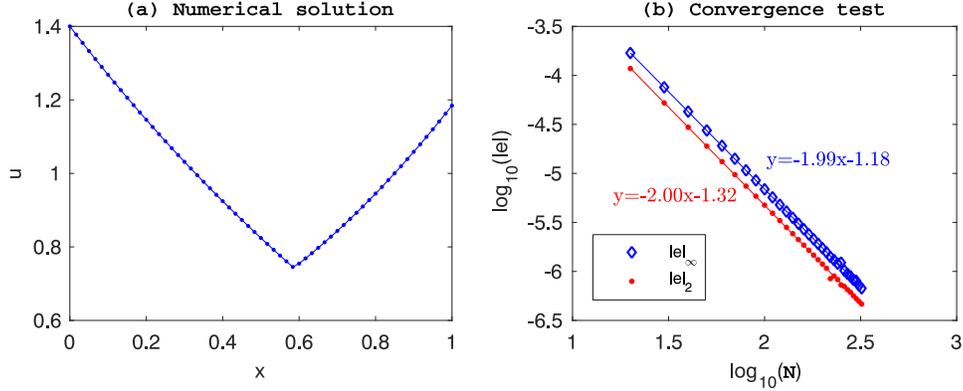}
    \caption{Example 1D-1. (a) Numerical solution with number of grid points $N=61$. (b) Error $\|e\|$ as a function of number of grid points $N$, as a log-log plot including slope of its linear fit.}\label{1d2ndord}
\end{figure}

\subsection{Two dimensions}\label{2DEXs}
The following 2D and 3D examples are tested on some rectangular domain $\Omega$ where $\Omega$ will be divided into $\Omega^+$ and $\Omega^-$ by an interface $\Gamma$. 
It will sometimes be convenient to describe the interface $\Gamma$ in terms of a
level-set function $\phi(x)$ as
$\Gamma=\{{\bf x}\in\Omega:\phi({\bf x})= 0\}$,
where the two sets $\Omega^+$ and $\Omega^-$ can be described as
$\Omega^+=\{{\bf x}\in\Omega:\phi({\bf x})>0\}$
and 
$\Omega^-=\{{\bf x}\in\Omega:\phi({\bf x})< 0\}$.
The coefficients $\beta$ are assumed to be smooth in both $\Omega^+$ and $\Omega^-$, but may have a jump across the interface $\phi$. The piecewise smooth $\beta$ in $\Omega^+$ and $\Omega^-$ will be denoted by $\beta^+$ and $\beta^-$, respectively. As a consequence, the solution $u$ may be discontinuous across $\phi$, but is $\mathscr{C}^2$ in both $\Omega^+$ and $\Omega^-$, and will similarly be denoted by $u^+$ and $u^-$, respectively.

\subsubsection{Example 2D-1: Constant coefficient.}

In this example, we take $\beta$ be a piecewise constant function with
$\beta^-=2$ and $\beta^+=1$,
and the interface is a circle described by the level set function 
$\phi(x,y)=(x-0.5)^2+(y-0.5)^2-0.25^2$.
The solution is
$u^-=\exp(-x^2-y^2)$, $u^+=0$, with $f^-=8(x^2+y^2-1)
\exp( -x^2-y^2 )$, $f^-=0$, 
on the domain $\Omega=[0,1]\times[0,1]$.
Second order convergence can be seen in figure (\ref{2dCirEx1Sol}b).

\begin{figure}[!htbp]
    \centering
        \includegraphics[width=\textwidth]{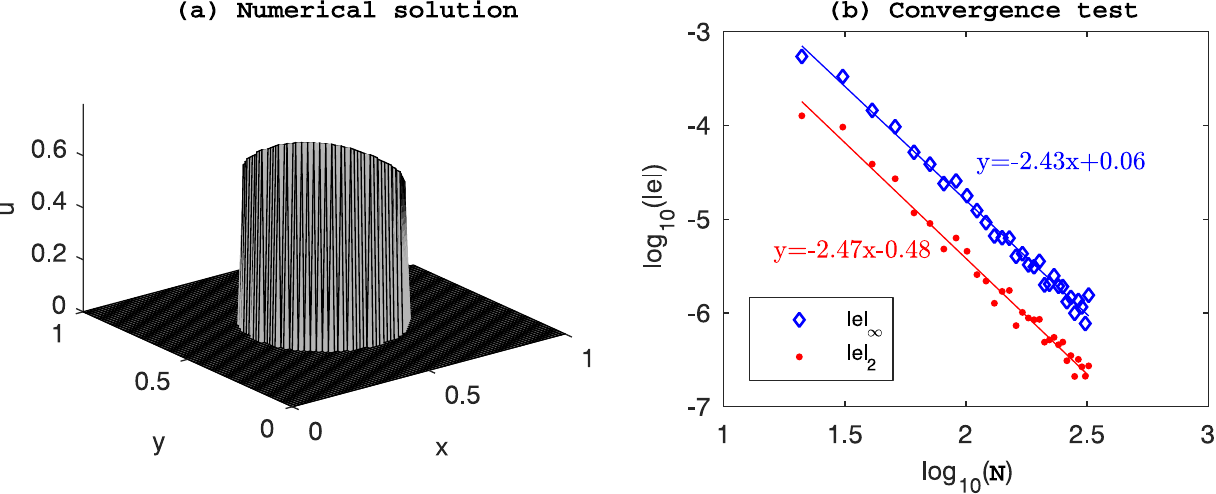}
    \caption{Example 2D-1: constant coefficient. (a): Numerical solution, $N=81$. (b): Error $\|e\|$ as a function of number of grid points in each coordinate direction, $N$, as a log-log plot including slope of its linear fit.}\label{2dCirEx1Sol}
\end{figure}

\subsubsection{Example 2D-2: Variable coefficient.}

The next example we take $\beta$ to be a piecewise smooth function with 
$\beta^-=x^2+y^2+1,$ and $\beta^+=1$
with the same domain and level set function as the previous example. The solution is $u^-=\exp(x^2+y^2),$ and $u^+=\exp(-x^2-y^2)$ and source term is $f^-=4( \beta^-(x^2+y^2+1) + (x^2+y^2) )\exp(x^2+y^2)$, $f^+=4(x^2+y^2-1)exp(-x^2-y^2)$. Error analysis is presented in figure (\ref{2dCirEx4Sol}b).

\begin{figure}[!htbp]
    \centering
        \includegraphics[width=\textwidth]{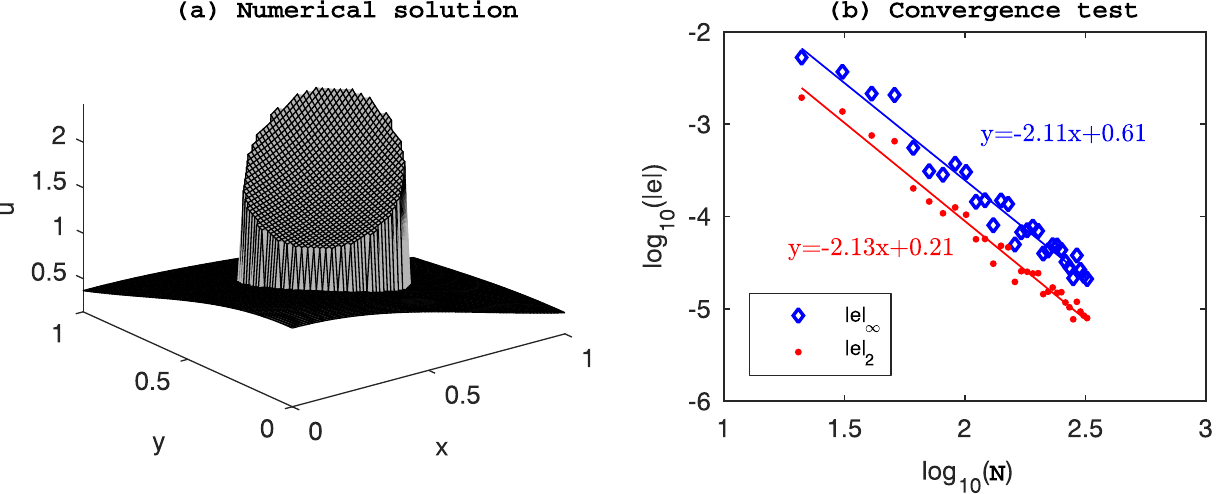}
    \caption{Example 2D-2: variable coefficient. (a): Numerical solution, $N=81$. (b): Error $\|e\|$ as a function of number of grid points in each coordinate direction, $N$, as a log-log plot including slope of its linear fit.}\label{2dCirEx4Sol}
\end{figure}

\subsubsection{Example 2D-3: Variable coefficient.}\label{sec:smoothstar}

With the same solution $u$ in example 2, this example is computed on on a domain $\Omega=[-1,1]\times[-1,1]$, with $\beta^-=x^2+y^2+1$ and $\beta^+=\sqrt{x^2+y^2+2}$, the corresponding $f^-=4( \beta^-(x^2+y^2+1) + (x^2+y^2) )\exp(x^2+y^2)$, $f^+=( 4\beta^+(x^2+y^2-1) - 2(x^2+y^2)/\sqrt{x^2+y^2+2} )\exp(-x^2+y^2)$. The interface is parameterized by
\begin{equation}
\left\{
\begin{array}{llll}
x(t)=0.02\sqrt{5}+(0.5+0.2\sin(5t))\cos(t),\\
y(t)=0.02\sqrt{5}+(0.5+0.2\sin(5t))\sin(t),
\end{array}
\right.
\end{equation}
with $t\in[0,2\pi]$. Second order convergence is demonstrated in figure (\ref{2dSSEx5Sol}b).

\begin{figure}[!htbp]
    \centering
        \includegraphics[width=\textwidth]{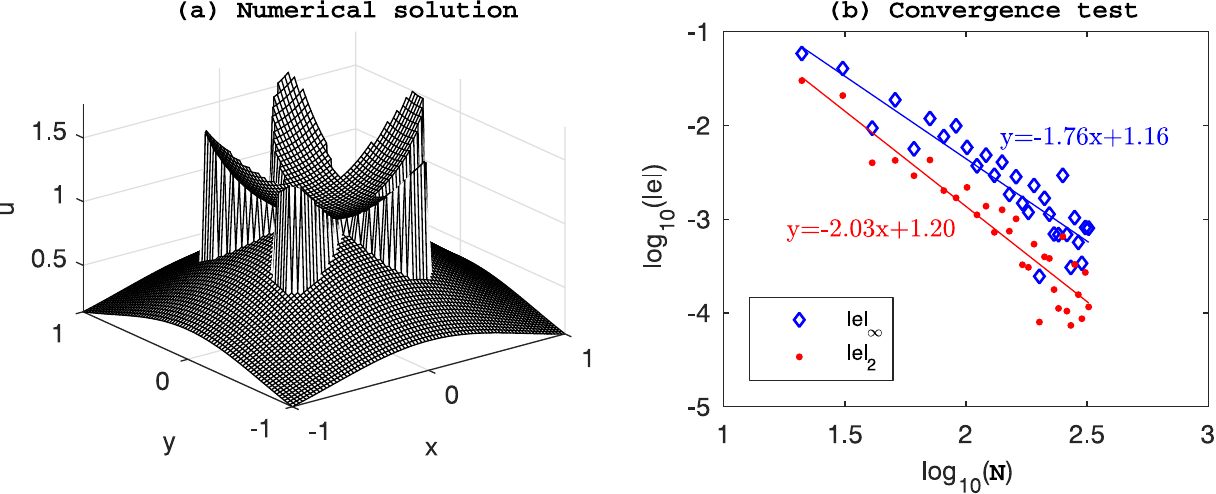}
    \caption{Example 2D-3: variable coefficient. (a): Numerical solution, $N=81$ (b): Error $\|e\|$ as a function of number of grid points in each coordinate direction, $N$, as a log-log plot including slope of its linear fit.}\label{2dSSEx5Sol}
\end{figure}

\subsubsection{Example 2D-4: High-Contrast coefficient cases}\label{sec:2dhctrs}
A series of tests were conducted on the large coefficient ratios, either $\beta^+/\beta^-\ll 1$ or $1\ll \beta^+/\beta^-$. Here we test with $u^-=\exp(x^2+y^2),$ and $u^+=\exp(-x^2-y^2)$ with a circular interface as in Example 2D-1, and $(\beta^+,\beta^-)=(0.02,1)$ and $(20,1)$. Second order convergence can still be obtained (see figure~\ref{2dCirHctrs}).

\begin{figure}[!htbp]
\centering
\includegraphics[width=\textwidth]{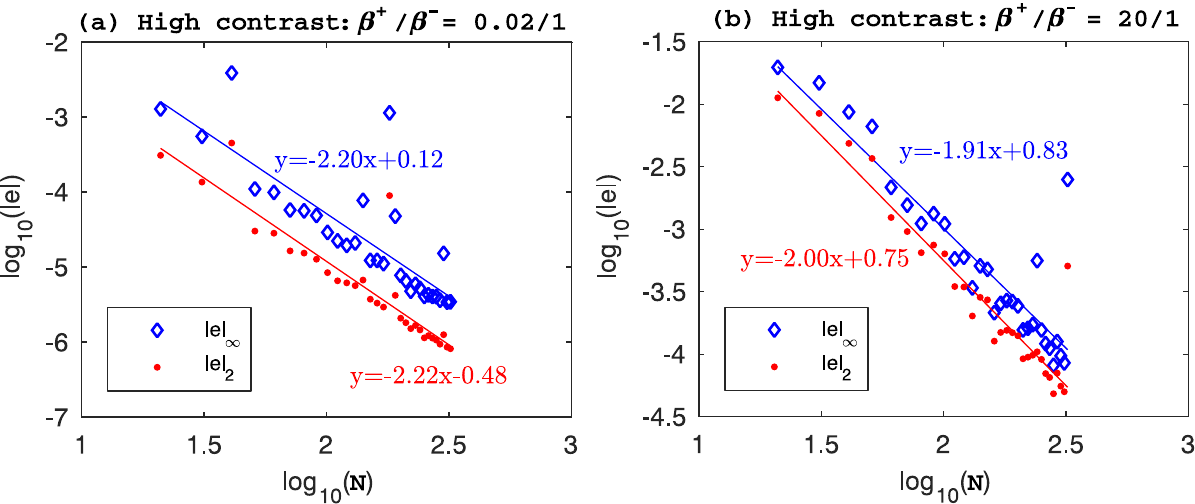}
\caption{Error plots for high contrast case, Example 2D-4. (a) $\beta^+/\beta^-=0.02/1$ (b) $\beta^+/\beta^-=20/1$.}
\label{2dCirHctrs}
\end{figure}

\subsection{Three dimensions}\label{3DEXs}

\subsubsection{Example 3D-1: Variable coefficient with spherical interface.}

On the domain $\Omega=[0,1]\times[0,1]\times[0,1]$, where $\Omega$ is divided into $\Omega^+$ and $\Omega^-$ by a sphere centered at $(0.5,0.5,0.5)$ with radius $0.25$.
The variable coefficients $\beta$ in equation (\ref{3deqn}) are 
$\beta^-=10+\sin(xy+z)$ and $\beta^+=10+\cos(x+yz)$, with solution $u^-=\exp(x^2+y^2+z^2)$ and $u^+=0$ and $f^-=(4\beta^-(x^2+y^2+z^2+3/2)+(4xy+2z)\cos(xy+z))\exp(x^2+y^2+z^2)$, $f^+=0$. See figure \ref{3DSphere} for 
the geometry of the spherical interface and second-order convergence in $L^2$.

\begin{figure}[!htbp]
    \centering
        \includegraphics[width=\textwidth]{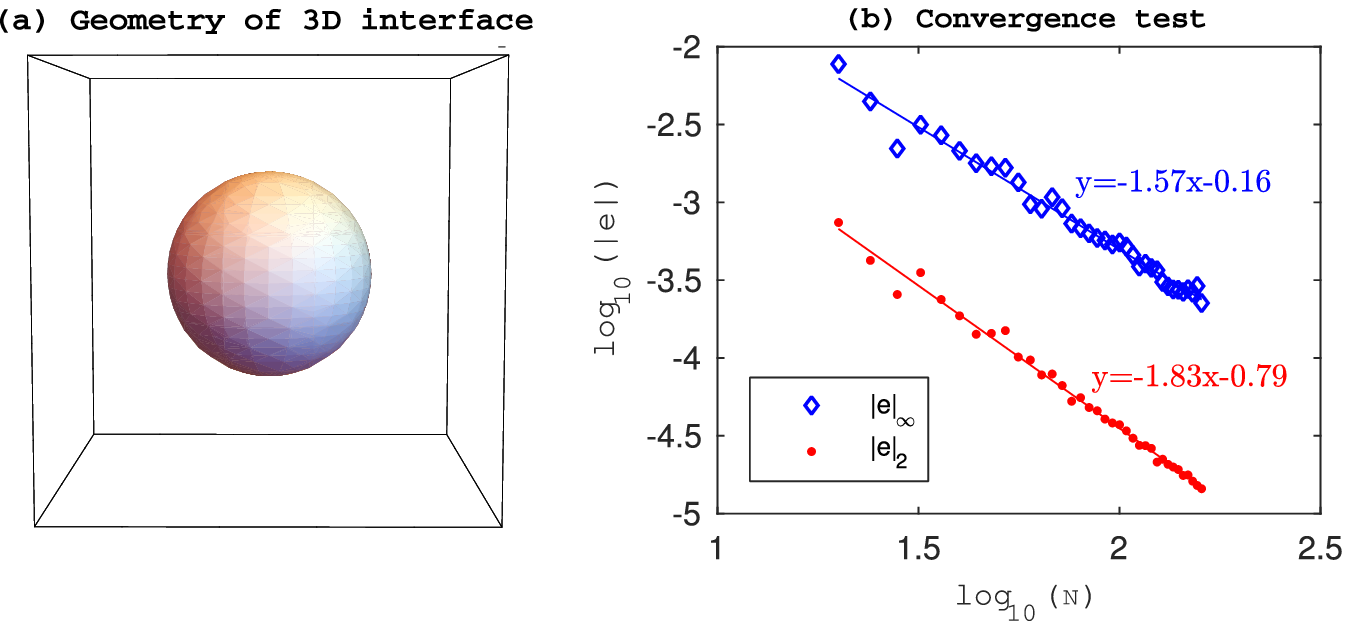}
    \caption{Example 3D-1: variable coefficient with spherical interface. (a): Geometry of the interface. (b): Error $\|e\|$ as a function of number of grid points in each coordinate direction, $N$, as a log-log plot including slope of its linear fit.}\label{3DSphere}
\end{figure}

\subsubsection{Example 3D-2. Variable coefficient with torus interface.}\label{sec:torus}

For the same $\beta$, $u$ and $f$ in example 1, we test this iterative method on $\Omega=[-1,1]\times[-1,1]\times[-1,1]$ with a toroid interface described by the level set function $\phi(x,y,z)=(x^2 + y^2 + z^2 +R^2-r^2 )^2-4R^2( x^2 + y^2 )$, where $R=0.501+\sqrt{2}/10$,  $r=0.251$.  
The geometry of the interface and second-order convergence in $L^\infty$ and $L^2$ are in figure \ref{3DTorus}.

\begin{figure}[!htbp]
    \centering
        \includegraphics[width=\textwidth]{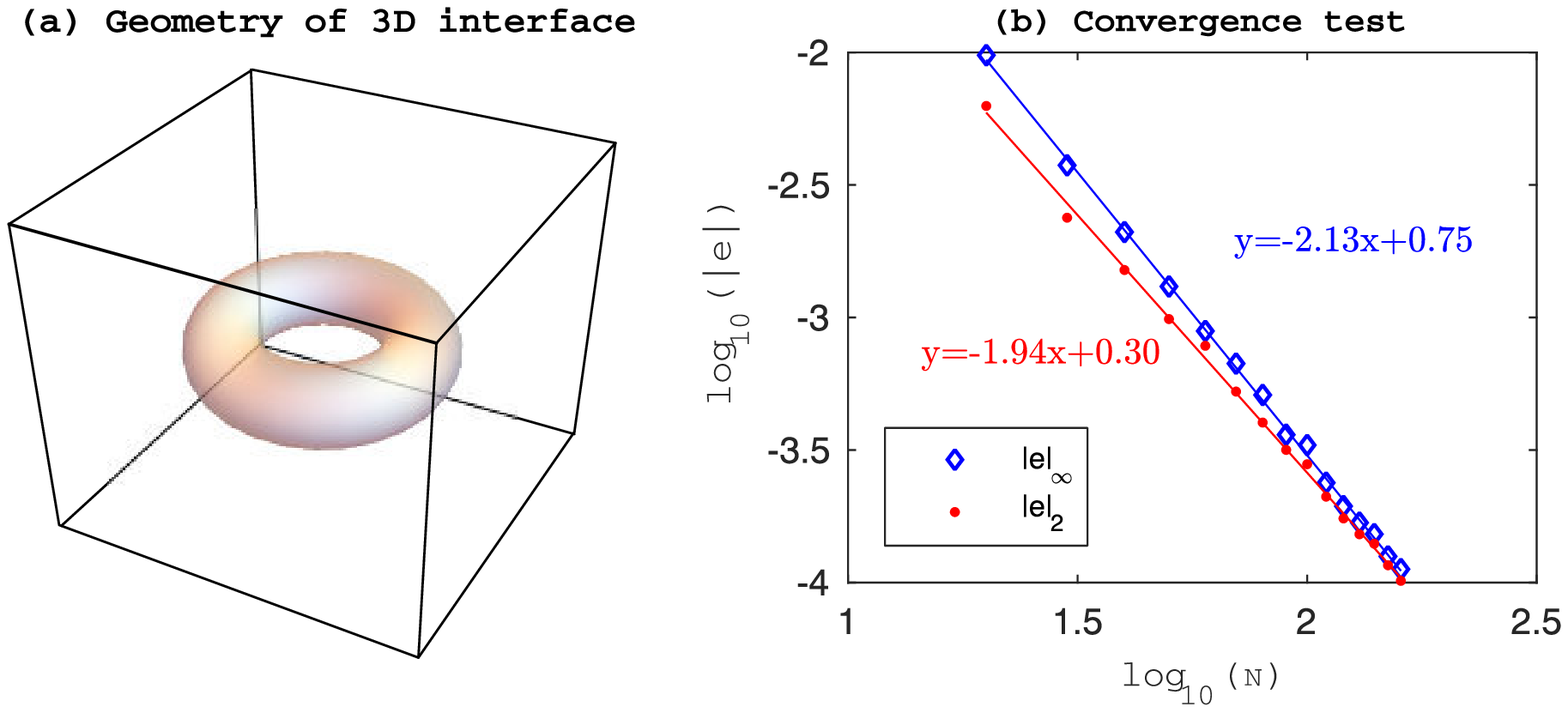}
    \caption{Example 3D-2: variable coefficient with torus interface. (a): Numerical solution. (b): Error $\|e\|$ as a function of number of grid points in each coordinate direction, $N$, as a log-log plot including slope of its linear fit.}\label{3DTorus}
\end{figure}

\section{Greater efficiency via alternative stopping criteria}
\label{sec:stopping}


\subsection{Iteration counts for standard stopping criterion}





In most of the cases shown above, the number of iterations required to reach the stopping criterion is small, which makes this iterative method efficient, as demonstrated in figure \ref{fig:23D_Niter}.  More specifically, approximately 10-20 iterations are used in 2D cases, and approximately 5-10 iterations in the 3D cases. For high contrast cases (example {\it 2D-4}), the number of iterations becomes larger (approximately 50-150, as seen in figure \ref{fig:HctrsNiter}),
but the number of iterations is essentially independent of the number of grid points.

These examples demonstrate that the present method may be
practical and efficient for time-dependent problems where the elliptic solver is needed at every time step. Below we discuss possibilities of further reducing the iterations counts through alternative stopping criteria -- e.g., by using a small, fixed number of iterations in~\S\ref{sec:fixiteration}, and propose some other feasible stopping criteria in~\S\ref{sec:altstop}.

\begin{figure}[!htbp]
\centering
\includegraphics[width=\textwidth]{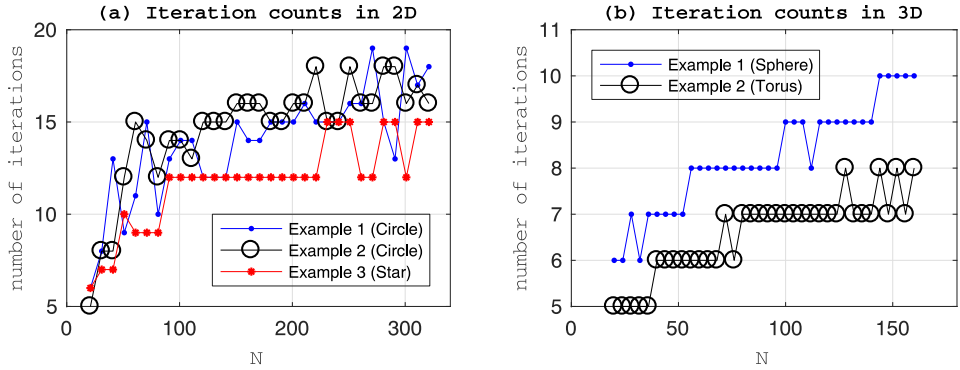}
\caption{Number of iterations for (a) 2D examples and (b) 3D examples.}
\label{fig:23D_Niter}
\end{figure}

\begin{figure}[!htbp]
\centering
\includegraphics[width=0.5\textwidth]{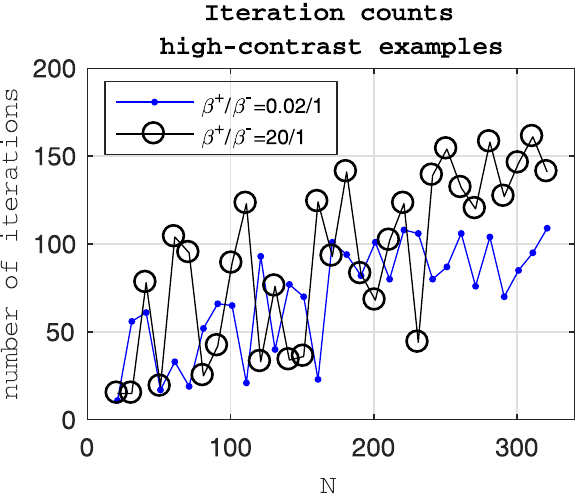}

\includegraphics[width=\textwidth]{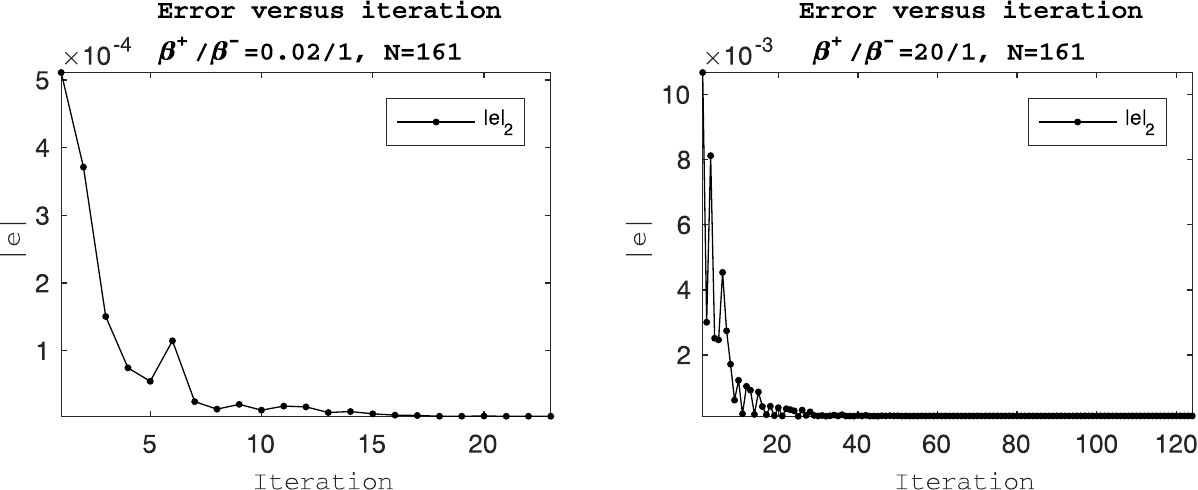}
\caption{Errors and iterations for the high-contrast cases from \S\ref{sec:2dhctrs}. 
Top: Number of iterations as a function of the number of grid points in each coordinate direction, $N$.
Bottom: $L^2$ error as a function of iterations, for $N=161$.}
\label{fig:HctrsNiter}
\end{figure}


\subsection{Greater efficiency via a small, fixed number of iterations}
\label{sec:fixiteration}

In most cases, the accuracy improves tremendously after only a few iterations; in other words, the latter iterations make only small modifications to the solution in order to satisfy the stopping criterion. Therefore, in practice, we may speed up this numerical method by using a fixed number of iterations without losing too much accuracy. Figure~\ref{5iterations} shows results of both 2D and 3D examples with only a small number of iterations (five), which still show second-order accuracy. 

\begin{figure}[!htbp]
\centering
\includegraphics[width=\textwidth]{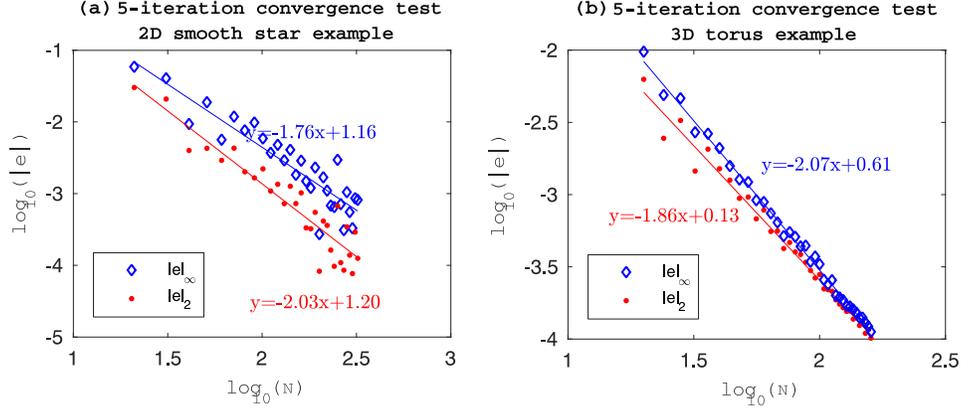}
\caption{Error as a function of number of grid points in each coordinate direction, $N$, using a fixed number of iterations (five) for more efficient computations. (a) Smooth star example in~\S\ref{2DEXs}. (b) Torus example in~\S\ref{3DEXs}.}
\label{5iterations}
\end{figure}

\subsection{Other stopping criteria}
\label{sec:altstop}
Several other stopping criteria were also tested,
beyond the standard criterion from \S\ref{sec:2d-iterative},
by using different combinations of criteria for the smallness of the differences
$u^{[k]}_{d}=\|\mathbf{u}^{[k+1]}-\mathbf{u}^{[k]}\|_\infty$
and/or
$F^{[k]}_{d}=\|\mathbf{F}^{[k+1]}-\mathbf{F}^{[k]}\|_\infty$.
A promising criterion may be to stop when $u^{[k]}_d<h^2$,
without enforcing any smallness criterion on $F^{[k]}_d$;
in some tests, this led to second-order accuracy with fewer
iterations, although we have not yet tested this criterion on a wide array of cases.

\section{Comparisons with formulations of other methods}
\label{sec:comparisons}

In this section we compare the present formulation with the formulations of other methods
\cite{wiegmann2000explicit,b04,rw03,lai2008simple,mnr11,mnr17},
to add to the comparisons with the GFM \cite{lfk00,ls03mc} and IIM \cite{ll94,li96,fk01,li01,dil03}
that were described above in \S\ref{1DFiniteDiff}.

In \cite{wiegmann2000explicit}, another approach had been taken to obtain a symmetric operator;
the derivation used Taylor series expansions and derivatives of jump conditions,
which can be somewhat complex compared to the simple derivations of the present paper
that mainly involve centered finite difference formulas.
Note that the present method and the method of \cite{wiegmann2000explicit} are, in fact,
distinct.  As one difference, in the 1D versions of the two methods,
the method of \cite{wiegmann2000explicit} has a non-symmetric operator in 1D,
whereas the method of the present paper has a symmetric operator in 1D.
Also, the method \cite{wiegmann2000explicit} utilizes a discretization of the
standard Laplacian operator, whereas the present method maintains the symmetry of
the elliptic operator that includes $\beta$.

In \cite{b04}, an interesting approach was proposed which, like the present method, involves a symmetric 
operator and an iterative method to determine an adjusted forcing. 
The derivation is somewhat complex in that it is a version of the IIM
and therefore uses Taylor series and derivatives of jump conditions.
The derivation is presented in 2D, but no 3D results are presented.
Also, their iterative procedure does not produce a first-order-accurate solution
at the first iteration, and therefore it is likely to require a very large number of 
iterations (as possibly indicated by their very small relaxation parameter).
The number of iterations, however, are not reported, and the iterative methods and
stopping criterion are not described in detail. 
In contrast, in the present paper, the first iteration is essentially the GFM,
and the simple finite-difference formulation allows for efficient setup and computation
even in 3D.

In \cite{lai2008simple}, following \cite{rw03}, another interesting approach 
is used to obtain a symmetric operator with corrections to the right-hand side.
The method is implemented in 2D, but no 3D results are presented.
Also, the method is presented for the standard Laplacian operator,
not for the case of discontinuous and/or spatially varying coefficient $\beta(\mathbf{x})$.


Another interesting method called the correction-function method 
has been developed by building on the GFM and
computing a corrected forcing function to achieve higher-order accuracy
\cite{mnr11,mnr17}.
In this method, the corrected forcing function is not derived explicitly;
instead, the corrected forcing function is shown to satisfy a certain new PDE,
and the new PDE is solved numerically to determine
the corrected forcing function.
The method has been demonstrated to achieve second-order and even fourth-order accuracy,
although it has not yet been implemented for 3D problems 
and it has only been developed for cases with constant coefficients and piecewise-constant coefficients.
It is similar to the method of the present paper in that both methods
seek to compute corrections to the GFM; the present paper's method perhaps offers a simpler
formulation (involving only one-dimensional finite differences) 
and simpler implementation for 3D problems.

\section{Conclusions}
\label{sec:conclusions}

In this article, a simple numerical scheme is proposed to obtain second-order accuracy in solving the Poisson equation with sharp interfaces.  One important contribution is a simple derivation that mainly involves centered finite difference formulas, with less reliance on the Taylor series expansions and derivatives of jump conditions used in typical immersed interface method derivations. The derivation here preserves the symmetry of the differential operator, and the method is formulated on a Cartesian grid. The accuracy of the method is proved rigorously in 1D and verified numerically in 2D and 3D. The three-dimensional problems are relatively easy to set up due to the method's simple derivation.

An iterative procedure was used for solving 2D or 3D problems, and
the desired second-order accuracy can be obtained with only a small, fixed number 
of iterations (typically 5), which makes this method efficient, even in 3D. 
In the future it would be interesting to investigate other algorithmic choices;
for instance, perhaps an
iterative method could be designed that requires an even
smaller number (e.g., 2 or 3) of iterations, or perhaps the method could be successful 
if the iterated correction terms 
were instead written as part of the left-hand-side linear operator,
in which case the symmetry of the operator is lost but the non-symmetric system could possibly be solved
without the need for the outer iterations introduced in the present paper.
Also, here we did not make a great effort to optimize the algorithms for cases with
high-contrast coefficients, which require higher iteration counts, 
but such an effort would be interesting to pursue in the future.

The proposed method may be applied to solving time-dependent problems
that require the solution of an elliptic PDE at each time step --
for example, the heat equation with interfaces or multiphase flow problems
\cite{glmp09,sfso98,ss17}.
In such applications, the present method could be used with 
any characterization of the interface (level set, Lagrangian markers, etc.),
and the interface could have a location and shape that evolves in time.



\section*{Acknowledgments}

The research of S.N.S. is partially supported by a Sloan Research Fellowship and
NSF grant AGS-1443325.
The authors thank J. T. Beale and A. Donev for helpful comments.

\appendix

\section{2D discretization with two interface crossings}
\label{app:2D-two-crossings}

In this appendix, it is shown how to formulate the finite difference method
in a case that is more general than in \S\ref{sec:fin-diff-2d}.

Suppose the interface crosses the stencil of point $(x_i,y_j)$ in two places,
as shown in Fig.~\ref{fig:AppendixStencil}. 
The crossing between $(x_i,y_j)$ and $(x_{i+1},y_j)$ is as in \S\ref{sec:fin-diff-2d},
and now a new, second crossing is present between
$(x_i,y_j)$ and $(x_i, y_{j+1})$. Accordingly, define $\zeta=(y_{j+1}-y_J)/\Delta y$, 
where $(x_i, y_J)\in \Gamma$, and assume $(x_i, y_j)\in \Omega^-$ 
and $(x_i, y_{j+1})\in \Omega^+$. 

\begin{figure}[!htbp]
\centering
\includegraphics[width=0.5\textwidth]{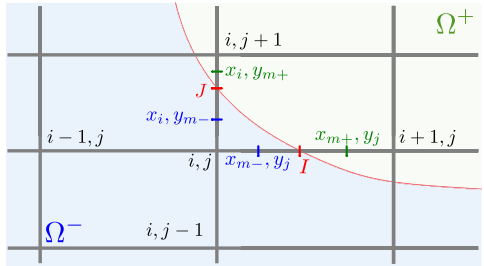}
\caption{Non-standard point $(x_i,y_j)$ with interface crossing the stencil in both $x$ and $y$ directions.}
\label{fig:AppendixStencil}
\end{figure}

To obtain a finite difference method with a symmetric operator in this case, 
start by writing the 1D formula from (\ref{symmetric_uxx}) as
\begin{equation}
S^x u=(\beta u_x)_x\cdot(2-\theta)/2+F_{cor}^x+O(\Delta x),
\end{equation}
where $S^x$ is the symmetric finite difference operator and
$F_{cor}^x$ is the correction term.
A similar formula can be derived for a symmetric finite difference operator
in the $y$ direction:
\begin{equation}
S^y u=(\beta u_y)_y\cdot(2-\zeta)/2+F_{cor}^y+O(\Delta y).
\end{equation}
Summing up the two leads to 
\begin{equation}
S^x u + S^y u = f -(\beta u_x)_x\cdot\theta/2-(\beta u_y)_y\cdot\zeta/2+F_{cor}^x+F_{cor}^y +O(\Delta x)+O(\Delta y),
\label{eqn:2d-general-appendix}
\end{equation}
which is the desired formula. 
Also note that the derivation in 3D follows the same simple principles 
by including the addition of a third component for $S^z u$.

Written out in detail, (\ref{eqn:2d-general-appendix}) takes the form
\begin{eqnarray}\notag
&& \frac{1}{(\Delta x)^2}\left(\beta(x_{i-\frac{1}{2}},y_j)\cdot u_{i-1,j}-\left(\beta(x_{i-\frac{1}{2}},y_j)+\hat\beta\right)u_{i,j}+\hat\beta\cdot u_{i+1,j}\right)\\[1ex]
&& +\frac{1}{(\Delta y)^2}\left(\beta(x_i,y_{j-\frac{1}{2}})\cdot u_{i,j-1}-\left(\beta(x_i,y_{j-\frac{1}{2}})+\tilde\beta\right)u_{i,j}+\tilde\beta\cdot u_{i,j+1}\right)\\[1ex]\notag
&&
=f_{i,j}-(\beta u_x)_x(x_i,y_j)\cdot\frac{\theta}{2}-(\beta u_y)_y(x_i,y_j)\cdot\frac{\zeta}{2}+F^x_{cor}+F^y_{cor}+O(\Delta x),
\label{symmetric_2d_uxx_uyy}
\end{eqnarray}
where $\hat \beta$ is the same as ($\ref{eqn:betahat}$) and
\begin{equation}
\tilde \beta=\frac{\beta(x_i,y_{m+})\cdot\beta(x_i,y_{m-})}{(1-\zeta)\cdot\beta(x_i,y_{m+})+\zeta\cdot\beta(x_i,y_{m-})},
\end{equation}
with midpoints $y_{m+}=(y_J+y_{j+1})/2$ and $y_{m-}=(y_j+y_{J})/2$, and
\begin{equation}
\begin{aligned}
F^x_{cor}
&=
\frac{\hat\beta \theta}{\beta(x_{m+},y_j)\Delta x}\Bigg\{\frac{\beta(x_{m+},y_j)a(x_I,y_j)}{\theta \Delta x}+[\beta u_x]\hfill
\\
&+\Big(\theta\cdot\left( f-(\beta u_y)_y\right)(x_{I+},y_j)+(1-\theta)\cdot\left( f-(\beta u_y)_y\right)(x_{I-},y_j)\Big)\frac{\Delta x}{2}\Bigg\},
\label{F^x_{cor}rection}
\end{aligned}
\end{equation}

\begin{equation}
\begin{aligned}
F^y_{cor}
&=
\frac{\tilde\beta \zeta}{\beta(x_i,y_{m+})\Delta y}\Bigg\{\frac{\beta(x_i,y_{m+})a(x_i,y_J)}{\zeta \Delta y}+[\beta u_y]\hfill
\\
&+\Big(\zeta\cdot\left( f-(\beta u_x)_x\right)(x_i,y_{J+})+(1-\zeta)\cdot\left( f-(\beta u_x)_x\right)(x_i,y_{J-})\Big)\frac{\Delta y}{2}\Bigg\},
\label{F^y_{cor}rection}
\end{aligned}
\end{equation}


where $[\beta u_x]=[\beta u_n]n^1-[\beta u_\tau]n^2$ and $[\beta u_y]=[\beta u_\tau]n^1+[\beta u_n]n^2$.

Several variations could also used. For instance, on the right-hand side of
(\ref{symmetric_2d_uxx_uyy}),
one may replace $(\beta u_x)_x$ by $f-(\beta u_y)_y$, 
or one may replace $(\beta u_y)_y$ by $f-(\beta u_x)_x$.
Similar replacements could be made in (\ref{F^x_{cor}rection}) and (\ref{F^y_{cor}rection}).
For our numerical tests, we used the $(\beta u_y)_y$ based version: 
$S^x u + S^y u = f\cdot(2-\theta)/2+(\beta u_y)_y\cdot(\theta-\zeta)/2+F_{cor}^x+F_{cor}^y$.

\section{Relaxation}\label{relaxation}

As discussed in \S\ref{sec:2d-iterative},
Picard iteration works well in many cases, but we found that it sometimes diverges.
For this reason, as our standard iterative scheme, we instead use a 
simple relaxation scheme to bypass this difficulty and guarantee 
that the iterative scheme stops.
The idea behind the relaxation scheme is to update the forcing term as
\begin{equation}
\mathbf{F}^{[k]}=\alpha_k \mathbf{F}^{[T_k]}+(1-\alpha_k)\mathbf{F}^{[k-1]},
\end{equation}
which is a mixture between the previous forcing $\mathbf{F}^{[k-1]}$
and the temporary forcing $\mathbf{F}^{[T_k]}$ that would have been used 
if a Picard update would have been followed. 
The parameter $\alpha_k$ is chosen to guarantee that $\mathbf{u}^{[k+1]}$
is not too far away from $\mathbf{u}^{[k]}$.

One cycle of the relaxation scheme goes as follows.
Suppose $\mathbf{u}^{[k]}$ was computed by solving 
$A{\mathbf{u}^{[k]}=\mathbf{F}^{[k-1]}}$, 
and we now want to compute the next iteration. 
With $\mathbf{u}^{[k]}$, compute the temporary right-hand-side $\mathbf{F}^{[T_k]}$ 
by following the Picard update procedure from \S\ref{sec:2d-iterative}.
A temporary solution $\mathbf{u}^{[T_{k+1}]}$ is then obtained 
by solving $A\mathbf{u}^{[T_{k+1}]} =\mathbf{F}^{[T_k]}$. 
Now the parameter $\alpha_k$ is determined to  
guarantee that $\mathbf{u}^{[k+1]}$ is not too far away from $\mathbf{u}^{[k]}$;
to this end, define the ratio
$r_k=\|\mathbf{u}^{[T_{k+1}]}-\mathbf{u}^{[k]} \| / \|\mathbf{u}^{[k]}-\mathbf{u}^{[k-1]}\|$.
If this ratio is small ($r_k < 1$), then there is no need for relaxation
and we set $\alpha_k=1$.
If this ratio is large ($r_k \ge 1$), then we set
$\alpha_k=\rho/r_k$, where $\rho$ is a 
preselected factor between $0$ and $1$. In practice, 
we pick $\|\cdot\|=\|\cdot \|_\infty$ and $\rho$ to be between 0.9 and 0.99.
With this relaxation scheme for the forcing $\mathbf{F}^{[k]}$,
the solution is likewise updated as
$\mathbf{u}^{[k+1]}=\alpha_k \mathbf{u}^{[T_{k+1}]}+(1-\alpha_k)\mathbf{u}^{[k]}$, 
as a mixture of the previous solution estimate 
$\mathbf{u}^{[k]}$
and the temporary solution estimate $\mathbf{u}^{[T_{k+1}]}$ that would have been used 
if a Picard update would have been followed. 

The differences
$u^{[k]}_{d}=\|\mathbf{u}^{[k+1]}-\mathbf{u}^{[k]}\|_\infty$
and
$F^{[k]}_{d}=\|\mathbf{F}^{[k+1]}-\mathbf{F}^{[k]}\|_\infty$
are guaranteed to be decreasing as $k$ increases
if this relaxation procedure is followed.
Specifically, the relaxation procedure leads to either $u_d^{[k]}=r_k u_d^{[k-1]}$ 
(if $r_k<1$) or $u_d^{[k]}=\rho u_d^{[k-1]}$ (if $r_k\ge 1$).
Therefore,
$u_d^{[k]}$ is decreasing in $k$ and hence the stopping criterion will be met 
in a finite number of iterations.
Note that this stopping criterion, 
based on $\|\mathbf{u}^{[k+1]}-\mathbf{u}^{[k]}\|$,
does not guarantee that the relaxation procedure's iterate $\mathbf{u}^{[k+1]}$ 
is actually close to the exact solution; 
nevertheless, one would expect that it should be at least a 
better estimate than the first iterate $\mathbf{u}^{[1]}$,
which is the first-order accurate GFM solution; and in practice we find from
the examples in \S\ref{sec:examples} that the iterations terminate at a
second-order accurate solution.

\bibliographystyle{spmpsci}      
\bibliography{sambib}   

\begin{thebibliography}{10}
\providecommand{\url}[1]{{#1}}
\providecommand{\urlprefix}{URL }
\expandafter\ifx\csname urlstyle\endcsname\relax
  \providecommand{\doi}[1]{DOI~\discretionary{}{}{}#1}\else
  \providecommand{\doi}{DOI~\discretionary{}{}{}\begingroup
  \urlstyle{rm}\Url}\fi

\bibitem{bl06}
Beale, T., Layton, A.: On the accuracy of finite difference methods for
  elliptic problems with interfaces.
\newblock Communications in Applied Mathematics and Computational Science
  \textbf{1}(1), 91--119 (2006)

\bibitem{betal10}
Bedrossian, J., Von~Brecht, J.H., Zhu, S., Sifakis, E., Teran, J.M.: A second
  order virtual node method for elliptic problems with interfaces and irregular
  domains.
\newblock J. Comput. Phys. \textbf{229}(18), 6405--6426 (2010)

\bibitem{b04}
Berthelsen, P.A.: A decomposed immersed interface method for variable
  coefficient elliptic equations with non-smooth and discontinuous solutions.
\newblock J. Comput. Phys. \textbf{197}(1), 364--386 (2004)

\bibitem{ccg11}
Crockett, R.K., Colella, P., Graves, D.T.: A cartesian grid embedded boundary
  method for solving the poisson and heat equations with discontinuous
  coefficients in three dimensions.
\newblock J. Comput. Phys. \textbf{230}(7), 2451--2469 (2011)

\bibitem{dil03}
Deng, S., Ito, K., Li, Z.: Three-dimensional elliptic solvers for interface
  problems and applications.
\newblock J. Comput. Phys. \textbf{184}(1), 215--243 (2003)

\bibitem{fetal13i}
Fai, T.G., Griffith, B.E., Mori, Y., Peskin, C.S.: Immersed boundary method for
  variable viscosity and variable density problems using fast
  constant-coefficient linear solvers {I}: Numerical method and results.
\newblock SIAM J. Sci. Comput. \textbf{35}(5), B1132--B1161 (2013)

\bibitem{fetal14ii}
Fai, T.G., Griffith, B.E., Mori, Y., Peskin, C.S.: Immersed boundary method for
  variable viscosity and variable density problems using fast
  constant-coefficient linear solvers {II}: theory.
\newblock SIAM J. Sci. Comput. \textbf{36}(3), B589--B621 (2014)

\bibitem{fk01}
Fogelson, A.L., Keener, J.P.: Immersed interface methods for {N}eumann and
  related problems in two and three dimensions.
\newblock SIAM J. Sci. Comput. \textbf{22}(5), 1630--1654 (2001)

\bibitem{glmp09}
Griffith, B.E., Luo, X., McQueen, D.M., Peskin, C.S.: Simulating the fluid
  dynamics of natural and prosthetic heart valves using the immersed boundary
  method.
\newblock Int. J. Appl. Mech. \textbf{1}(01), 137--177 (2009)

\bibitem{gp05}
Griffith, B.E., Peskin, C.S.: On the order of accuracy of the immersed boundary
  method: Higher order convergence rates for sufficiently smooth problems.
\newblock J. Comput. Phys. \textbf{208}(1), 75--105 (2005)

\bibitem{hetal12}
Hellrung, J.L., Wang, L., Sifakis, E., Teran, J.M.: A second order virtual node
  method for elliptic problems with interfaces and irregular domains in three
  dimensions.
\newblock J. Comput. Phys. \textbf{231}(4), 2015--2048 (2012)

\bibitem{jly06}
Ji, H., Lien, F.S., Yee, E.: An efficient second-order accurate cut-cell method
  for solving the variable coefficient {P}oisson equation with jump conditions
  on irregular domains.
\newblock Int. J. Numer. Meth. Fluids \textbf{52}(7), 723--748 (2006)

\bibitem{kbgd16}
Kallemov, B., Bhalla, A.P.S., Griffith, B.E., Donev, A.: An immersed boundary
  method for rigid bodies.
\newblock Comm. App. Math. and Comp. Sci. \textbf{11}(1), 79--141 (2016).
\newblock \doi{10.2140/camcos.2016.11.79}

\bibitem{kfl00}
Kang, M., Fedkiw, R.P., Liu, X.D.: A boundary condition capturing method for
  multiphase incompressible flow.
\newblock J. Sci. Comput. \textbf{15}(3), 323--360 (2000)

\bibitem{lp00}
Lai, M.C., Peskin, C.S.: An immersed boundary method with formal second-order
  accuracy and reduced numerical viscosity.
\newblock J. Comput. Phys. \textbf{160}(2), 705--719 (2000)

\bibitem{lai2008simple}
Lai, M.C., Tseng, H.C.: A simple implementation of the immersed interface
  methods for stokes flows with singular forces.
\newblock Computers \& Fluids \textbf{37}(2), 99--106 (2008)

\bibitem{ll03}
Lee, L., LeVeque, R.J.: An immersed interface method for incompressible
  {N}avier--{S}tokes equations.
\newblock SIAM J. Sci. Comput. \textbf{25}(3), 832--856 (2003)

\bibitem{ll94}
Leveque, R.J., Li, Z.: The immersed interface method for elliptic equations
  with discontinuous coefficients and singular sources.
\newblock SIAM J. Numer. Anal. \textbf{31}(4), 1019--1044 (1994)

\bibitem{li96}
Li, Z.: A note on immersed interface method for three-dimensional elliptic
  equations.
\newblock Comput. Math. Appl. \textbf{31}(3), 9--17 (1996)

\bibitem{li01}
Li, Z., Ito, K.: Maximum principle preserving schemes for interface problems
  with discontinuous coefficients.
\newblock SIAM J. Sci. Comput. \textbf{23}(1), 339--361 (2001)

\bibitem{ll01}
Li, Z., Lai, M.C.: The immersed interface method for the navier--stokes
  equations with singular forces.
\newblock J. Comput. Phys. \textbf{171}(2), 822--842 (2001)

\bibitem{lfk00}
Liu, X.D., Fedkiw, R.P., Kang, M.: A boundary condition capturing method for
  {P}oisson's equation on irregular domains.
\newblock J. Comput. Phys. \textbf{160}(1), 151--178 (2000)

\bibitem{ls03mc}
Liu, X.D., Sideris, T.: Convergence of the ghost fluid method for elliptic
  equations with interfaces.
\newblock Math. Comput. \textbf{72}(244), 1731--1746 (2003)

\bibitem{shlui12}
Lui, S.H.: Numerical analysis of partial differential equations, vol. 102.
\newblock John Wiley \& Sons (2012)

\bibitem{mnr11}
Marques, A.N., Nave, J.C., Rosales, R.R.: A correction function method for
  {P}oisson problems with interface jump conditions.
\newblock J. Comput. Phys. \textbf{230}(20), 7567--7597 (2011).
\newblock \doi{10.1016/j.jcp.2011.06.014}

\bibitem{mnr17}
Marques, A.N., Nave, J.C., Rosales, R.R.: High order solution of {P}oisson
  problems with piecewise constant coefficients and interface jumps.
\newblock J. Comput. Phys. \textbf{335}, 497--515 (2017).
\newblock \doi{10.1016/j.jcp.2017.01.029}

\bibitem{mp08}
Mori, Y., Peskin, C.S.: Implicit second-order immersed boundary methods with
  boundary mass.
\newblock Comput. Meth. Appl. Mech. Eng. \textbf{197}(25), 2049--2067 (2008)

\bibitem{p72}
Peskin, C.S.: Flow patterns around heart valves: a numerical method.
\newblock J. Comput. Phys. \textbf{10}(2), 252--271 (1972)

\bibitem{p77jcp}
Peskin, C.S.: Numerical analysis of blood flow in the heart.
\newblock J. Comput. Phys. \textbf{25}(3), 220--252 (1977)

\bibitem{p02}
Peskin, C.S.: The immersed boundary method.
\newblock Acta Numerica \textbf{11}, 479--517 (2002)

\bibitem{rw03}
Russell, D., Wang, Z.J.: A cartesian grid method for modeling multiple moving
  objects in 2d incompressible viscous flow.
\newblock Journal of Computational Physics \textbf{191}(1), 177--205 (2003)

\bibitem{ss17}
Smith, L.M., Stechmann, S.N.: Precipitating quasigeostrophic equations and
  potential vorticity inversion with phase changes.
\newblock J. Atmos. Sci. \textbf{74}, 3285--3303 (2017).
\newblock \doi{10.1175/JAS-D-17-0023.1}

\bibitem{sfso98}
Sussman, M., Fatemi, E., Smereka, P., Osher, S.: An improved level set method
  for incompressible two-phase flows.
\newblock Computers \& Fluids \textbf{27}(5), 663--680 (1998)

\bibitem{sshoz07}
Sussman, M., Smith, K.M., Hussaini, M.Y., Ohta, M., Zhi-Wei, R.: A sharp
  interface method for incompressible two-phase flows.
\newblock J. Comput. Phys. \textbf{221}(2), 469--505 (2007)

\bibitem{wiegmann2000explicit}
Wiegmann, A., Bube, K.P.: The explicit-jump immersed interface method: finite
  difference methods for pdes with piecewise smooth solutions.
\newblock SIAM Journal on Numerical Analysis \textbf{37}(3), 827--862 (2000)

\bibitem{xw06jcp}
Xu, S., Wang, Z.J.: An immersed interface method for simulating the interaction
  of a fluid with moving boundaries.
\newblock J. Comput. Phys. \textbf{216}(2), 454--493 (2006)

\bibitem{xw06siam}
Xu, S., Wang, Z.J.: Systematic derivation of jump conditions for the immersed
  interface method in three-dimensional flow simulation.
\newblock SIAM J. Sci. Comput. \textbf{27}(6), 1948--1980 (2006)

\end{thebibliography}


\mbox{}\\
\noindent Received: / Accepted:\\

\noindent Chung-Nan Tzou: \email{ctzou@wisc.edu}\\
Department of Mathematics, University of Wisconsin--Madison, Madison, WI 53706, United States.\\

\noindent Samuel N. Stechmann: \email{stechmann@wisc.edu}\\
Department of Mathematics, University of Wisconsin--Madison, Madison, WI 53706, United States.
\end{document}